\documentclass[runningheads]{llncs}
\usepackage{amsmath,amssymb}
\usepackage{bbm}
\usepackage{tikz}
\usetikzlibrary{arrows.meta, decorations.pathmorphing}
\usepackage{graphicx}
\usepackage[hyphens]{url}
\usepackage[bookmarks=true]{hyperref}

\newcommand{\RR}{\mathbb{R}}
\newcommand{\VV}{\mathcal{V}}

\newcommand{\vv}{\boldsymbol{v}}
\newcommand{\xx}{\boldsymbol{x}}
\renewcommand{\aa}{\boldsymbol{a}}
\newcommand{\bb}{\boldsymbol{b}}
\newcommand{\cc}{\boldsymbol{c}}
\newcommand{\aalpha}{\boldsymbol{\alpha}}

\newcommand{\1}{\mathbbm{1}}

\newcommand{\red}{\textcolor{red}}

\begin{document}

\title{Slack Ideals in Macaulay2}

\author{Antonio Macchia\thanks{Supported by the Einstein Foundation Berlin under Francisco Santos grant EVF-2015-230.}\inst{1}\orcidID{0000-0001-9680-0560} \and
Amy Wiebe\inst{1}}

\authorrunning{A. Macchia and A. Wiebe}

\institute{Fachbereich Mathematik und Informatik, Freie Universit\"at Berlin,\\ Arnimallee 2, 14195 Berlin, Germany\\
\email{macchia.antonello@gmail.com, w.amy.math@gmail.com}}

\maketitle              

\begin{abstract}
Recently Gouveia, Thomas and the authors introduced the slack realization space, a new model for the realization space of a polytope. It represents each polytope by its slack matrix, the matrix obtained by evaluating each facet inequality at each vertex. Unlike the classical model, the slack model naturally mods out projective transformations. It is inherently algebraic, arising as the positive part of a variety of a saturated determinantal ideal, and provides a new computational tool to study classical realizability problems for polytopes.
We introduce the package \texttt{SlackIdeals} for \textit{Macaulay2}, that provides methods for creating and manipulating slack matrices and slack ideals of convex polytopes and matroids. Slack ideals are often difficult to compute. To improve the power of the slack model, we develop two strategies to simplify computations: we scale as many entries of the slack matrix as possible to one; we then obtain a reduced slack model combining the slack variety with the more compact Grassmannian realization space model. This allows us to study slack ideals that were previously out of computational reach.

As applications, we show that the well-known Perles polytope does not admit rational realizations and prove the non-realizability of a large quasi-simplicial sphere.

\keywords{Polytopes \and Slack matrices \and Slack ideals \and Matroids.}
\end{abstract}

\section{Introduction}

Slack matrices of polytopes are nonnegative real matrices whose entries express the slack of a vertex in a facet inequality. In particular, the zero pattern of a slack matrix encodes the vertex-facet incidence structure of the polytope. Slack matrices have found remarkable use in the theory of extended formulations of polytopes: Yannakakis \cite{Y91} proved that the extension complexity of a polytope is equal to the nonnegative rank of its slack matrix.

More generally, one can define the slack matrix of a matroid by computing the slacks of the ground set vectors in the hyperplanes of the matroid.

If $P$ is $d$-dimensional polytope, replacing all positive entries in the slack matrix with distinct variables, one obtains a new sparse generic matrix $S_P(\xx)$, called the \textit{symbolic slack matrix} of $P$. Then we define the \textit{slack ideal} $I_P$ of $P$ as the ideal of all $(d+2)$-minors of $S_P(\xx)$, saturated with respect to the product of all variables in $S_P(\xx)$.

Slack ideals were introduced for polytopes in \cite{GPRT17}, where it was also noted that they could be used to model the realization space of a polytope. The details of this realization space model and further properties of the slack ideal were studied in \cite{GMTWfirstpaper}, \cite{GMTWsecondpaper} and \cite{GMWthirdpaper}. An analogous realization space model for matroids was introduced in \cite{BW19}.

In this paper, we describe the \texttt{Macaulay2} \cite{M2} package \texttt{SlackIdeals.m2}, that is available at \url{https://bitbucket.org/macchia/slackideals/src/master/SlackIdeals.m2}. It provides methods to define and manipulate slack matrices of polytopes, matroids, polyhedra, and cones; obtain a slack matrix directly from the Gale transform of a polytope; compute the symbolic slack matrix and the slack ideal from a slack matrix; compute the graphic ideal of a polytope, the cycle ideal and the universal ideal of a matroid.

Slack ideal computations are often out of computational reach. Therefore we develop two techniques to speed up and simplify computations. First, we suitably set to one as many entries of the slack matrix as possible. One can compute the slack ideal of this dehomogenized slack matrix and then rehomogenize the resulting ideal (see Proposition \ref{PROP:rehomIdeal}). The new ideal coincides with the original slack ideal if the latter is radical. Second, we obtain a reduced slack matrix by keeping the columns of a set of facets $F$ that contains a flag (a maximal chain in the face lattice of P) and such that the facets not in $F$ are simplicial. Combining these two strategies, we have a powerful tool for the study of hard realizability questions. As applications, we show that the well-known Perles polytope does not admit rational realizations and prove the non-realizability of a large quasi-simplicial sphere.

\section{Slack matrices and slack ideals}

Given a collection of points $V = \{\vv_1,\ldots, \vv_n\}\subset\RR^d$ and a collection of (affine) hyperplanes $H = \{\{\xx\in\RR^d: b_i-\aalpha_i^\top\xx = 0\} : i=1\ldots f\}$  we can define a \textit{slack matrix of the pair} $(V,H)$ by
\[
S_{(V,H)} = \begin{bmatrix} \1 & \vv_1 \\ \vdots & \vdots \\ \1 & \vv_n \end{bmatrix} \begin{bmatrix} b_1 & \cdots & b_f \\ \aalpha_1 & \cdots & \aalpha_f \end{bmatrix} \in \RR^{n\times f}.
\]

If $P$ is a $d$-polytope, we take $V = \textup{vert}(P)$ and $H$ to be the set of facet defining hyperplanes. Then $S_P = S_{(V,H)}$. When coordinates $V$ are given for the vectors of a matroid~$M$, they are always assumed to be an affine configuration which gets homogenized to form the matroid; in particular, this means that if $V = \textup{vert}(P)$, then the associated matroid is the matroid of the polytope $P$. The hyperplanes are taken to be all hyperplanes of $M$, and then $S_M = S_{(V,H)}$.

{\renewcommand{\baselinestretch}{0.9}
\begin{verbatim}
i1 : needsPackage "SlackIdeals";
i2 : V = {{0,0},{0,1},{1,1},{1,0}};
-- Compute the slack matrix of P=conv(V)
i3 : slackMatrix(V)
o3 = | 0 1 0 1 |
     | 1 0 0 1 |
     | 0 1 1 0 |
     | 1 0 1 0 |
-- Compute the slack matrix of matroid of V
i4 : slackMatrix(V, Object=>"matroid")
o4 = | -1 -1 0 -1 0  0  |
     | -1 0  1 0  1  0  |
     | 0  1  1 0  0  -1 |
     | 0  0  0 -1 -1 -1 |
\end{verbatim}
}
The \texttt{slackMatrix} command also takes a pre-computed matroid, polyhedron or cone object as input.

Another way to compute the slack matrix of a polytope is from its Gale transform using the command \texttt{slackFromGaleCircuits}. Let $G$ be a matrix with real entries whose columns are the vectors of a Gale transform of a polytope $P$. A slack matrix of $P$ is computed by finding the minimal positive circuits of $G$, see \cite[Section 5.4]{G03}.
Alternatively, the command \texttt{slackFromGalePlucker} applies the maps of \cite[Section~5]{GMWthirdpaper} to fill a slack matrix with Pl\"ucker coordinates of the Gale transform.

The slack matrices of a few specific polytopes and matroids of theoretical importance are built-in, using the command \texttt{specificSlackMatrix}.

\medskip
The \textit{symbolic slack matrix} can be obtained by replacing the nonzero entries of a slack matrix by distinct variables; that is,
\[
[S_{(V,H)}(\xx)]_{i,j} = \begin{cases} 0 & \text{ if } \vv_i\in H_j \\ x_{i,j} & \text{ if } \vv_i\notin H_j \end{cases}.
\]

From this sparse generic matrix we obtain the \textit{slack ideal} as the saturation of the ideal of its $(d+2)$-minors by the product of all variables in $S_{(V,H)}(\xx)$:
\[
I_{(V,H)} = \langle (d+2)-\textup{minors of } S_{(V,H)}(\xx)\rangle : \left(\prod_{j=1}^f\prod_{i : \vv_i\notin H_j} x_{i,j}\right)^\infty.
\]

Given a (symbolic) slack matrix of a $d$-polytope, $(d+1)$-dimensional cone, or rank $d+1$ matroid, we can compute the associated slack ideal, specifying $d$ as an input. Unless we pass variable names as an option, the function labels the variables consecutively by rows with a single index starting from $1$:

{\renewcommand{\baselinestretch}{0.9}
\begin{verbatim}
-- Compute slack ideal of d-polytope P=conv(V)
i10 : V = {{0,0},{0,1},{1,1},{1,0}};
i11 : slackIdeal(2, slackMatrix(V)) -- here d=2
o11 = ideal(x x x x  - x x x x )
             0 3 5 6    1 2 4 7
\end{verbatim}
}

We get the same result if we compute \texttt{slackIdeal(2,V)}, giving only the list of vertices of a $d$-polytope or ground set vectors of a matroid instead of a slack matrix. {We also get the same result with \texttt{slackIdeal(V)}, but the computation is faster if you provide $d$ as an argument.}
As optional argument, one can choose the object to be set as \texttt{"polytope"}, \texttt{"cone"}, or \texttt{"matroid"} (default is \texttt{Object=>"polytope"}).

To a polytope or matroid we can also associate a specific toric ideal, known as the {\em graphic} or {\em cycle ideal}, respectively. These ideals are important in the classification of certain projectively unique polytopes \cite{GMTWsecondpaper} and matroids \cite{BW19}, and can be computed using the commands \texttt{graphicIdeal} and \texttt{cycleIdeal}.

In \cite[Section 4]{GMWthirdpaper} it is shown that a slack matrix can be filled with Pl\"ucker coordinates of a matrix formed from the vertex coordinates of a polytope (or extreme ray generators of a cone or ground set vectors of a matroid). This idea is the basis for the reduction technique described in \cite[Section 6] {GMWthirdpaper} and Section~\ref{S.Reduction}. The Grassmannian section ideal of a polytope is also defined and shown to cut out exactly  a set of representatives of the slack variety that are constructed in this way \cite[Section 4.1]{GMWthirdpaper}.
The command \texttt{grassmannSectionIdeal} computes this section ideal given a set of vertices of a polytope and the indices of vertices that span each facet.

\section{On the dehomogenization of the slack ideal} 

Let $P$ be a polytope and $S_P$ its slack matrix. We define the \textit{non-incidence graph} $G_P$ as the bipartite graph whose vertices are the vertices and facets of $P$, and whose edges are the vertex-facet pairs of $P$ such that the vertex is not on the facet. This graphic structure provides a systematic way to scale a maximal number of entries in $S_P$ to~$1$, as spelled out in \cite[Lemma~5.2]{GMTWsecondpaper}. In particular, we may scale the rows and columns of $S_P(\xx)$ so that it has ones in the entries indexed by the edges in a maximal spanning forest of the graph $G_P$. This can be done using \texttt{setOnesForest}, which outputs a sequence $(Y,F)$ where $Y$ is the scaled symbolic slack matrix and $F$ is the spanning forest used to scale $Y$.

{\renewcommand{\baselinestretch}{0.9}
\begin{verbatim}
i23 : V = {{0,0,0},{1,0,0},{0,1,0},{0,0,1},{1,0,1},{1,1,0}};
i24 : (Y, F) = setOnesForest(symbolicSlackMatrix(V)); Y
o24 = | 0 1 0   0    1 |
      | 1 0 0   0    1 |
      | 0 1 1   0    0 |
      | 1 0 x_7 0    0 |
      | 0 1 0   1    0 |
      | 1 0 0   x_11 0 |
\end{verbatim}}

This leads to a dehomogenized version of the slack ideal defined as follows. Given $S_P$ and a maximal spanning forest $F$ of $G_P$, let $S_P(\xx^F)$ be the symbolic slack matrix of $P$ with all the variables corresponding to edges in~$F$ set to $1$. Then the dehomogenized ideal, $I_P^F$, is the slack ideal of this scaled slack matrix:
$$I_P^F := \langle (d+2)-\textup{minors of }S_P(\xx^F)\rangle : \left(\prod \xx^F\right)^\infty.$$

It is natural to ask what is the relation between $I_P^F$ and the original slack ideal $I_P$. In particular, we might wish to know if we can recover the full slack ideal from $I_P^F$. From \cite[Lemma 5.2]{GMTWsecondpaper} we know that any slack matrix in $\mathcal V(I_P)$ (or, in fact, any point in the slack variety with all coordinates that correspond to $F$ being nonzero) can be scaled to a matrix in $\mathcal V(I^F_P)$. Conversely, it is clear that any point in $\mathcal V(I^F_P)$ can be thought of as a point in $\mathcal V(I_P)$. Thus, in terms of the varieties we have $\VV(I_P)^*/(\RR^v\times\RR^f) \cong \VV(I_P^F)^*,$ where $\VV(I)^*$ denotes the part of the variety where all coordinates are nonzero.

To see the algebraic implications of this, let us introduce the following rehomogenization process.
Notice that in the proof of \cite[Lemma 5.2]{GMTWsecondpaper}, we dehomogenize by following the edges of forest $F$ starting from some chosen root(s) and moving toward the leaves. The destination vertex of each edge tells us which row or column to scale, and the edge label is the variable by which we scale.
Now, given a polynomial in $I_P^F$, using the same forest and orientation we proceed in the reverse order: starting at the leaves, for each edge of the forest, we reintroduce the variable corresponding to it in order to rehomogenize the polynomial with respect to the row or column corresponding to the destination vertex of that edge.

\begin{example} Consider the slack matrix $S_P(\xx^F)$ of the triangular prism $P$ scaled
\noindent\begin{minipage}[c]{0.65\textwidth} according to forest $F$, pictured in Figure \ref{F:forest}. Then $I_P^F = \langle x_7-1,x_{11}-1\rangle$. So we can rehomogenize, for example, the element $x_7-x_{11}$ with respect to forest~$F$ as follows.\\ \indent First, consider the leaf corresponding to column 3. \end{minipage} \hspace{0.02\textwidth}
{\renewcommand{\arraystretch}{0.8}\begin{minipage}[c]{0.29\textwidth} \vspace{-10pt}
\[
S_P(\xx^F) \!=\! \begin{bmatrix} 0 & 1 & 0 & 0 & 1 \\ 1 & 0 & 0 & 0 & 1 \\ 0 & 1 & 1 & 0 & 0 \\ 1 & 0 & x_7 & 0 & 0 \\ 0 & 1 & 0 & 1 & 0 \\ 1 & 0 & 0 & x_{11} & 0 \end{bmatrix}
\]
\end{minipage}}\\[0.1cm]
Its edge is labeled with $x_5$, so we reintroduce that variable to the monomial $x_{11}$ since its degree in column 3 is currently $0$, while the degree of $x_7$ in that column is $1$. We continue this process until all the edges of $F$ have been used.
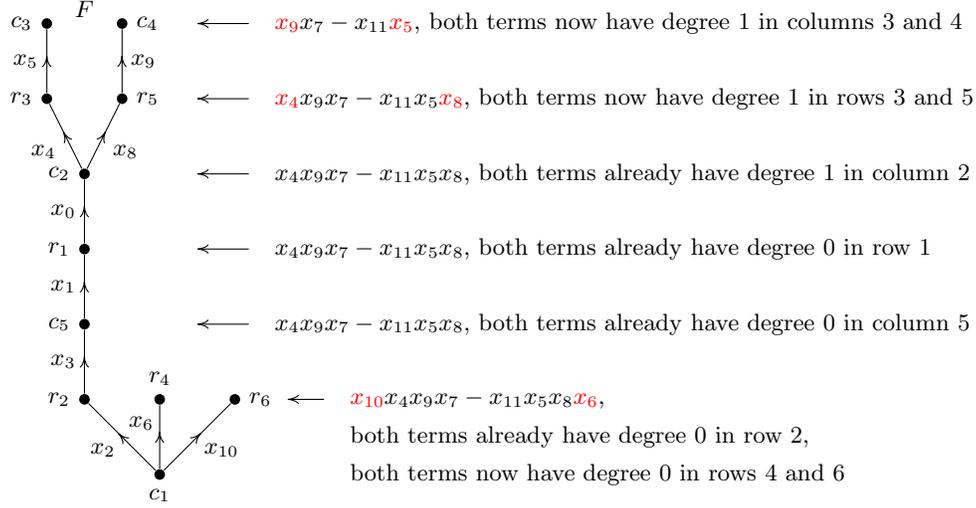
\begin{figure}
\scalebox{1}{\hspace{-15pt}
\begin{tikzpicture}
\newcommand{\midarrow}{\tikz \draw[-{>[length=1.5mm, width=1.5mm,angle'=45,open]}] (0,0) -- +(0.1,0);}
\draw (-1,0.2) node[] {$F$};
\draw (0,-6) node[circle, fill, inner sep = 0pt, minimum size = 4pt, label = {below:$c_1$}] (c1) {};
\draw (-1,-5) node[circle, fill, inner sep = 0pt, minimum size = 4pt, label = {left:$r_2$}] (r2) {};
\draw (0,-5) node[circle, fill, inner sep = 0pt, minimum size = 4pt, label = {above:$r_4$}] (r4) {};
\draw (1,-5) node[circle, fill, inner sep = 0pt, minimum size = 4pt, label = {right:$r_6$}] (r6) {};
\draw (-1,-4) node[circle, fill, inner sep = 0pt, minimum size = 4pt, label = {left:$c_5$}] (c5) {};
\draw (-1,-3) node[circle, fill, inner sep = 0pt, minimum size = 4pt, label = {left:$r_1$}] (r1) {};
\draw (-1,-2) node[circle, fill, inner sep = 0pt, minimum size = 4pt, label = {left:$c_2$}] (c2) {};
\draw (-1.5,-1) node[circle, fill, inner sep = 0pt, minimum size = 4pt, label = {left:$r_3$}] (r3) {};
\draw (-0.5,-1) node[circle, fill, inner sep = 0pt, minimum size = 4pt, label = {right:$r_5$}] (r5) {};
\draw (-1.5,0) node[circle, fill, inner sep = 0pt, minimum size = 4pt, label = {left:$c_3$}] (c3) {};
\draw (-0.5,0) node[circle, fill, inner sep = 0pt, minimum size = 4pt, label = {right:$c_4$}] (c4) {};
\draw (c1) -- node[sloped, label={[label distance=-6pt]below:$x_{2}$}]{\rotatebox{180}{\midarrow}}
(r2) --node[sloped,label={[label distance=-5pt]above:$x_{3}$}]{\midarrow}
(c5) --node[sloped,label={[label distance=-5pt]above:$x_{1}$}]{\midarrow}
(r1) --node[sloped,label={[label distance=-5pt]above:$x_{0}$}]{\midarrow}
(c2) --node[sloped,label={[label distance=-5pt]-110:$x_{4}$}]{\rotatebox{180}{\midarrow}}
(r3) --node[sloped,label={[label distance=-5pt]above:$x_{5}$}]{\midarrow} (c3);
\draw (c1) -- node[sloped,label={[label distance=-7pt]60:$x_{6}$}]{\midarrow} (r4);
\draw (c1) -- node[sloped,label={[label distance=-6pt]below:$x_{10}$}]{\midarrow} (r6);
\draw (c2) -- node[sloped,label={[label distance=-5pt]-70:$x_{8}$}]{\midarrow}
(r5) -- node[sloped,label={[label distance=-5pt]below:$x_{9}$}]{\midarrow} (c4);
\draw (1.3,0) node[label=right:{$\red{x_{9}}x_7-x_{11}\red{x_5}$, {\small both terms now have degree $1$ in columns 3 and 4}}] (p1) {};
\draw[-{>[length=1.5mm, width=1.5mm,angle'=45,open]}] (p1) -- (0.5,0);
\draw (1.3,-1) node[label=right:{$\red{x_4}x_{9}x_7-x_{11}x_5\red{x_8}$, {\small both terms now have degree $1$ in rows 3 and 5 }}] (p2) {};
\draw[-{>[length=1.5mm, width=1.5mm,angle'=45,open]}] (p2) -- (0.5,-1);
\draw (1.3,-2) node[label=right:{$x_4x_{9}x_7-x_{11}x_5x_8$, {\small both terms already have degree $1$ in column 2}}] (p3) {};
\draw[-{>[length=1.5mm, width=1.5mm,angle'=45,open]}] (p3) -- (0.5,-2);
\draw (1.3,-3) node[label=right:{$x_4x_{9}x_7-x_{11}x_5x_8$, {\small both terms already have degree $0$ in row 1}}] (p4) {};
\draw[-{>[length=1.5mm, width=1.5mm,angle'=45,open]}] (p4) -- (0.5,-3);
\draw (1.3,-4) node[label=right:{$x_4x_{9}x_7-x_{11}x_5x_8$, {\small both terms already have degree $0$ in column 5}}] (p5) {};
\draw[-{>[length=1.5mm, width=1.5mm,angle'=45,open]}] (p5) -- (0.5,-4);
\draw (2.3,-5) node[label=right:{$\red{x_{10}}x_4x_{9}x_7-x_{11}x_5x_8\red{x_6}$, }] (p6) {};
\draw (2.3,-5.5) node[label=right:{{\small both terms already have degree $0$ in row 2,}}] {};
\draw (2.3,-6) node[label=right:{{\small both terms now have degree $0$ in rows 4 and 6}}] {};
\draw[-{>[length=1.5mm, width=1.5mm,angle'=45,open]}] (p6) -- (1.7,-5);
\end{tikzpicture}
}
\caption{A spanning forest for the triangular prism} \label{F:forest}
\end{figure}
\end{example}

Call the resulting ideal $H({I}_P^F)$. By the tree structure, the rehomogenization process does indeed end with a polynomial that is homogeneous, as once we make it homogeneous for a row or column we never add variables in that row or column again. We now consider the effect of this rehomogenization on minors.

\begin{lemma}
Let $p$ be a minor of $S_P(\xx)$ and $p^F$ its dehomogenization by $F$. Then its rehomogenization $H(p^F)$ equals $p$ divided by the product of all variables in $F$ that divide $p$.
\label{LEM:minorhomog}
\end{lemma}

\begin{proof}
Note that all monomials in a minor have degree precisely one on every relevant row and column. In fact they can be interpreted as perfect matchings on the subgraph of $G_P$ corresponding to the $(d+2) \times (d+2)$ submatrix being considered. Let $\xx^{\aa}$ and $\xx^{\bb}$ be two distinct monomials in the minor, then their dehomogenizations are also distinct. To see this, note that if we interpret $\aa$ and $\bb$ as matchings, a common dehomogenization would be a common submatching $\cc$ of both, with all the remaining edges being in $F$. But $\aa \setminus \cc$ and $\bb \setminus \cc$ would then be distinct matchings on the same set of variables, hence their union contains a cycle, so they would not be both contained in the forest $F$.

Now note that when rehomogenizing a minor, we start with all degrees being zero or one for every row and column, and since we visit each node (corresponding to each of the rows/columns) exactly once by the tree structure, the degree of every row and column is at most one after homogenizing.
In the first step of rehomogenizing, we start with a leaf of~$F$, which means the variable $x_i$ labeling its edge is the only variable in the row or column corresponding to that leaf which was set to~1. Thus if any monomial of the minor has degree zero on that row or column, it must be because $x_i$ occurred in that monomial in the original minor.

Hence rehomogenizing will just add that variable to the monomials where it was originally present, with the exception of the case where it was present on all monomials, in which case there will be no need to add it, as the dehomogenized polynomial would be homogeneous (of degree 0) for that particular row/column.

All degrees remain 0 or 1 after this process, and now the node incident to the leaf we just rehomogenized corresponds to a row/column with exactly one variable that is still dehomogenized. Thus we can repeat the argument on the entire forest to find that each monomial rehomogenizes to itself divided by the variables that were originally present in all monomials of the minor.
\end{proof}

\begin{remark}
It is important to note that $H(I_P^F)$ is the ideal of {\em all elements} of $I_P^F$ rehomogenized. In general, this is different from the ideal generated by the rehomogenized generators of $I_P^F$. In the package, we rehomogenize the whole ideal by rehomogenizing the generators and saturating the resulting ideal by all the variables we just homogenized by.
\end{remark}

For example, let $V$ be the set of vertices of the triangular prism with spanning forest $Y$ as computed before, and let us compute the rehomogenized ideal $H(I_P^F)$.

{\renewcommand{\baselinestretch}{0.9}
\begin{verbatim}
i25 : HIF = rehomogenizeIdeal(3, Y, F)
\end{verbatim}
}
{\renewcommand{\baselinestretch}{0.9}
\begin{verbatim}
o25 = ideal (x x x x   - x x x x  , x x x x   - x x x x  ,
              4 7 9 10    5 6 8 11   0 3 9 10    1 2 8 11
      x x x x  - x x x x )
       0 3 5 6    1 2 4 7
\end{verbatim}
}
Notice that, in this case the rehomogenized ideal $H(I_P^F)$ equals the slack ideal~$I_P$.

\begin{example} Recall that the generators of $I_P^F$ for the triangular prism were $x_7-1$ and $x_{11}-1$, which rehomogenize to $x_1x_2x_4x_7-x_0x_3x_5x_6$ and $x_1x_2x_8x_{11}-x_0x_3x_9x_{10}$, respectively. However,
\[
\langle x_1x_2x_4x_7-x_0x_3x_5x_6, x_1x_2x_8x_{11}-x_0x_3x_9x_{10} \rangle \neq H(I_P^F).
\]
\end{example}

The relation between the rehomogenized ideal $H(I_P^F)$ and the original slack ideal is given in the following lemma. The proof relies on the key fact that the variety of the rehomogenized ideal is still the same as the slack variety that we started with.

\begin{proposition}
Given a spanning forest $F$ for the non-incidence graph of polytope $P$, the rehomogenization of its scaled slack ideal is an intermediate ideal between the slack ideal and its radical:
$I_P \subseteq H(I_P^F) \subseteq \sqrt{I_P}$. \label{PROP:rehomIdeal}
\end{proposition}

\begin{proof}
To prove the inclusion $I_P \subseteq H(I_P^F)$, note that $p \in I_P$ happens if and only if $\xx^{\aa} p \in J$ for some exponent vector $\aa$, where $J$ is the ideal generated by all $(d+2)$-minors of the symbolic slack matrix of $P$. Dehomogenizing we get $\xx^{\bb} p^F \in J^F$, which means $p^F$ is in the saturation of $J^F$ by the product of all variables, which is precisely the definition of $I_P^F$. From Lemma~\ref{LEM:minorhomog} it follows that $p\in H(I_P^F)$.

To prove that $H(I_P^F) \subseteq \sqrt{I_P}$, it is enough to show that any polynomial in $H(I_P^F)$ vanishes in the slack variety. By construction, any such polynomial must vanish on the points of the slack variety where the variables corresponding to the forest $F$ are nonzero, $\VV(I_P)\backslash \VV(\langle \xx^F\rangle)$. Thus, they vanish on the Zariski closure of that set. Considering the following containments,
$$\VV(I_P)\backslash\VV(\langle\xx\rangle) \subset \VV(I_P)\backslash \VV(\langle\xx^F\rangle) \subset \VV(I_P),$$
we get that this closure is exactly the slack variety since $\overline{\VV(I_P)\backslash\VV(\langle\xx\rangle)} = \VV(I_P:\left<\xx\right>^\infty) = \mathcal V(I_P)$.
\end{proof}

\begin{remark}
One would like to say that $I_P = H(I_P^F)$, and so far we have no counterexample for this equality, since it always holds if $I_P$ is radical, and we also have no examples of non-radical slack ideals.
\end{remark}

\section{Reduced slack matrices} \label{S.Reduction}

In general, computing the slack ideal may take a long time or be infeasible, especially if the dimension of the polytope is small compared to its number of vertices and facets. In some cases we can speed up this computation combining the slack and the Grassmannian realization space models \cite[Section 6]{GMWthirdpaper}. In fact, we do not need to work with the full slack matrix, since the essential information is contained into a sufficiently large submatrix.

We will see in Examples~\ref{EX:Perles} and~\ref{EX:sphere}, that slack ideals which we were not even able to compute (using personal computers) are now able to be calculated in a matter of a few seconds. To give an estimate of the improvement, computing the slack ideal of the full slack matrix in Example~\ref{EX:Perles} requires the computation of about $8.6 \cdot 10^9$ minors, whereas the reduced slack ideal only requires the computation of about $1.9 \cdot 10^4$ minors.

More precisely, let $P$ be a realizable polytope and $F$ be a set of facets of~$P$ such that $F$ contains a set of facets that can be intersected to form a flag in the face lattice of $P$ and all facets of $P$ not in $F$ are simplicial. We call a \textit{reduced slack matrix} for $P$ the submatrix, $S_F$, of $S_P$ consisting of only the columns indexed by $F$. Set $\VV_F$ to be the nonzero part of the slack variety $\VV(I_F)$.

If $\overline{\VV_F}$ is irreducible, then $\VV_F \times \mathbb C^h \cong \VV(I_P)^*$ are birationally equivalent, where~$h$ denotes the number of facets of $P$ outside $F$ \cite[Proposition 6.9]{GMWthirdpaper}.

\begin{example}
Let $P$ be the Perles projectively unique polytope with no rational realization coming from the point configuration in \cite[Figure 5.5.1, p. 93]{G03}. This is an $8$-polytope with $12$ vertices and $34$ facet and its symbolic slack matrix $S_P(\xx)$ is a $12 \times 34$ matrix with 120 variables.

Let $S_F$ be the following submatrix of $S_P$ whose 13 columns correspond to all the nonsimplicial facets of $P$:
{\renewcommand{\baselinestretch}{0.9}
\begin{verbatim}
i28 : S = specificSlackMatrix("perles1");
-- Checking that the first 13 columns of S indeed contain a flag
i29 : containsFlag(toList(0..12),S)
o29 = true
i30 : SF = reducedSlackMatrix(8, S, FlagIndices=>toList(0..12));
\end{verbatim}
}
The associated symbolic slack matrix is:
\[
S_F(\xx) =
\begin{bmatrix}
0 & 0 & 0 & x_0 & x_1 & x_2 & 0 & 0 & 0 & 0 & 0 & 0 & 0\\
0 & 0 & 0 & x_3 & 0 & 0 & x_4 & x_5 & x_6 & 0 & 0 & 0 & 0\\
0 & 0 & 0 & 0 & 0 & 0 & x_7 & 0 & 0 & x_8 & x_9 & 0 & 0\\
0 & 0 & 0 & 0 & x_{10} & 0 & 0 & 0 & 0 & 0 & 0 & x_{11} & x_{12}\\
0 & 0 & 0 & 0 & 0 & 0 & 0 & x_{13} & 0 & x_{14} & 0 & x_{15} & 0\\
x_{16} & 0 & 0 & 0 & 0 & 0 & 0 & 0 & 0 & 0 & x_{17} & 0 & x_{18}\\
0 & x_{19} & 0 & 0 & 0 & 0 & 0 & 0 & x_{20} & 0 & 0 & 0 & 0\\
0 & 0 & x_{21} & 0 & 0 & x_{22} & 0 & 0 & 0 & 0 & 0 & 0 & 0\\
x_{23} & 0 & 0 & x_{24} & 0 & 0 & 0 & x_{25} & 0 & 0 & 0 & 0 & 0\\
0 & x_{26} & 0 & 0 & x_{27} & 0 & 0 & 0 & 0 & x_{28} & 0 & 0 & 0\\
0 & 0 & x_{29} & 0 & 0 & 0 & x_{30} & 0 & 0 & 0 & 0 & 0 & x_{31}\\
0 & 0 & 0 & 0 & 0 & x_{32} & 0 & 0 & x_{33} & 0 & x_{34} & x_{35} & 0
\end{bmatrix}.
\]

Using \cite[Lemma 5.2]{GMTWsecondpaper}, we first set $x_i=1$ for $i = 0,3,4,5,6,7,8,9,12,14,15,16,\break 17,20,21,25,26,27,28,29,30,31,32,34$. The resulting scaled reduced slack ideal is:
\begin{gather*}
\langle \boldsymbol{x_{35}^2+x_{35}-1}, x_{33}-x_{35}-1, x_{24}-x_{35}, x_{23}-x_{35}, x_{22}-1,x_{19}-x_{35},\\
x_{18}-x_{35},x_{13}-x_{35}-1, x_{11}-x_{35},x_{10}-1,x_{2}-1,x_{1}-x_{35}-1 \rangle.
\end{gather*}
It follows that $x_{35}=\frac{-1 \pm \sqrt{5}}{2}$. Hence, $P$ does not admit rational realizations.
\label{EX:Perles}
\end{example}

\begin{example} \label{EX:sphere}
Let $P$ be the abstract polytope, labeled \#1963 in \cite{CS19}, with 14 vertices labeled $0, \dots, 6, a \dots, g$ and with 94 facets, $\{0,1,2,3,4,5,6\}, \{a,b,c,d,e,f,g\}$ and the other 92 listed in \cite[Table 3]{CS19}.

A reduced slack matrix for $P$ (where facets $F = \{F_0, F_1, F_2, F_3, F_4, F_{10}\}$ form a flag) with the maximum number of variables set to one is the following:
\[
S_F(\xx) = \begin{bmatrix}
0 & 1 & 0 & 0 & 0 & 0 \\
0 & x_{10} & 0 & 0 & 1 & 0 \\
0 & 1 & 0 & x_{18} & 0 & 0 \\
0 & 1 & x_{22} & x_{23} & 0 & 0 \\
0 & x_{33} & x_{34} & 0 & x_{35} & 1 \\
0 & x_{43} & x_{44} & x_{45} & x_{46} & 1 \\
0 & 1 & 0 & 1 & 1 & 1 \\
x_{68} & 0 & x_{69} & 0 & x_{70} & 1 \\
x_{76} & 0 & x_{77} & x_{78} & 0 & 1 \\
x_{85} & 0 & x_{86} & x_{87} & x_{88} & 1 \\
x_{95} & 0 & 0 & 1 & 0 & 0 \\
1 & 0 & 1 & 0 & x_{105} & 1 \\
x_{113} & 0 & x_{114} & x_{115} & x_{116} & 1 \\
x_{127} & 0 & x_{128} & x_{129} & x_{130} & 1
\end{bmatrix}
\]
Now we reconstruct the remaining columns of the slack matrix. We can then recursively determine the sign of each column, by looking for monomial entries and setting the sign of all the entries of that column so they are positive.

From this process, we get a collection of polynomials that must be simultaneously positive. In particular, from this matrix we get polynomials that imply the inequalities $x_{76} > x_{77} > x_{34} > 1$. Furthermore, we have degree 2 polynomials including $-x_{34}x_{76} + x_{34} + x_{76} - x_{77} > 0$. The first inequalities give us
\[
\frac{x_{76}-x_{77}}{x_{76}-1} < 1
\]
while the second gives
\[
\frac{x_{76} - x_{77}}{x_{76}-1} > x_{34}
\]
which is a contradiction to $x_{34}>1$. Thus we have found a subset of the entries of the slack matrix which cannot be simultaneously positive, so that $P$ is not realizable.

\end{example}

The previous example shows that the reduction process can be a powerful tool to show nonrealizability of large quasi-simplicial spheres.

\bigskip
\noindent \textbf{Acknowlegements.} We would like to thank Jo\~{a}o Gouveia for helping us with Section 3.

%
%
 \bibliographystyle{splncs04}
 \bibliography{mybibliography}
\end{document}